\documentclass[a4paper,12pt,reqno]{amsart}
\usepackage{amsmath,amsfonts,amssymb,amsthm,enumerate,multicol}
\usepackage{tikz,graphicx}
\usepackage{hyperref}
\usepackage{caption,enumitem,wasysym}
\usepackage{cleveref}
\usepackage{epstopdf}
\usepackage{float,wrapfig}
\usepackage[labelsep=space]{caption}
\usepackage[font=footnotesize]{caption}
\usepackage{xcolor}

\newtheorem{definition}{Definition}[section]
\newtheorem{theorem}{Theorem}[section]
\newtheorem{corollary}{Corollary}[section]

\newtheorem{lemma}{Lemma}[section]

\newtheorem{remark}{Remark}[section]

\allowdisplaybreaks

 \theoremstyle{plain}
\newtheorem{thm}{Theorem}[section]

\newtheorem{prop}[thm]{Proposition}

\DeclareMathOperator{\sgn}{sgn}

\usepackage{mathtools}

\usepackage{amsmath,amsfonts,amssymb,amsthm,enumerate,multicol}
\usepackage{tikz}
\usepackage{float}
\allowdisplaybreaks

\numberwithin{equation}{section}

\begin{document}
\begin{center}
\Large{\textbf{Well-posedness of the  discrete collision-induced breakage equation and various properties of solutions }}
\end{center}





\medskip
\medskip
\centerline{${\text{Mashkoor~ Ali$^*$}}$, ${\text{Ankik ~ Kumar ~Giri$^{\dagger}$}}$ and ${\text{Philippe~ Lauren\c{c}ot$^{\ddagger}$}}$}\let\thefootnote\relax\footnotetext{$^{\dagger}$Corresponding author. Tel +91-1332-284818 (O);  Fax: +91-1332-273560  \newline{\it{${}$ \hspace{.3cm} Email address: }}ankik.giri@ma.iitr.ac.in}
\medskip
{\footnotesize

  \centerline{ ${}^{}$  $*$,$\dagger$ Department of Mathematics, Indian Institute of Technology Roorkee,}
   \centerline{Roorkee-247667, Uttarakhand, India}

 \centerline{ ${}^{}$ $\ddagger$ Laboratoire de Math\'ematiques (LAMA) UMR 5127, Universit\'e Savoie Mont Blanc, CNRS,}
   \centerline{ F-73000, Chamb\'ery, France}
}

\bigskip

\begin{quote}
{\small {\em \bf Abstract.}  A discrete version of the nonlinear collision-induced breakage equation is studied. Existence of solutions is investigated for a broad class of unbounded collision kernels and daughter distribution functions, the collision kernel $a_{i,j}$ satisfiying $a_{i,j} \leq A i j$ for some $A>0$. More precisely, it is proved that, given suitable conditions, there exists at least one mass-conserving solution for all times. A result on the uniqueness of solutions is also demonstrated under reasonably general conditions. Furthermore, the propagation of moments, differentiability, and the continuous dependence of solutions are established, along with some invariance properties and the large-time behaviour of solutions.}
\end{quote}


\vspace{.3cm}



\section{Introduction}

Breakage, also known as fragmentation, is a basic process that describes the dissociation of particles that may occur in a variety of scientific and technical fields, including chemical process engineering, astrophysics, atmospheric science, and cellular biology. Depending on the particle breakage behaviour, the breakage process may be categorised into two kinds: The first is the \emph{linear breakage} which can happen spontaneously or as a result of external forces, and the second is the \emph{collision-induced nonlinear breakage} which takes place when two particles collide. One of the most effective approaches to characterising the kinetics of such phenomena is with the help of a rate equation which captures the evolution of the distribution of interacting clusters with respect to their sizes (or masses). In this article, we are interested in studying a mathematical model that governs the collision-induced breakage, which is often exploited to depict the raindrop breakup, cloud formation, and planet formation, see, for instance, \cite{LG 1976, SV 1972, SRC 1978}. The model under consideration here is known as the \emph{collision-induced breakage equation} or sometimes also referred to as the \emph{nonlinear fragmentation equation}. It is a nonlinear nonlocal equation featuring quadratic nonlinearities and it describes the time evolution of the mass (or size) distribution function of particles undergoing collision-induced fragmentation. In the so-called continuous case, the size (or mass) of each particle is denoted by a positive real number, whereas in the discrete case, the ratio of the mass of the basic building block (monomer) to the mass of a typical cluster is a positive integer and the size of a cluster is a finite multiple of the monomer's mass, i.e., \ a positive integer.

 In contrast to the linear (or spontaneous) fragmentation equation which has received a lot of attention since the pioneering works of Filippov \cite{FAF 61}, Kapur \cite{KAPUR 72}, McGrady \& Ziff \cite{MCG 87, ZRM 85}, see \cite{BLL 2019, BERTOIN 2006} and the references therein for a more detailed account, the nonlinear fragmentation equation has not been thoroughly investigated until recently. When the mass variable ranges in the set of positive real numbers, some particular cases are studied in the physical literature. In \cite{CHNG 90}, Cheng \& Redner study the asymptotic behaviour of a class of models in which a two-particle collision causes both particles to split into two equal halves. Three cases are considered: both particles split, only the largest particle splits, or only the smallest particle splits. They also show that some models can be mapped to the linear fragmentation equation after a change of time scale. This transformation is thoroughly used in \cite{EP 2007} to analyze the nonlinear fragmentation equation with product collision kernels and to discuss the existence and non-existence of solutions, along with the formation of singularities in finite time. Further insight in the dynamics of the models considered in \cite{CHNG 90} is provided by Krapivsky \& Ben-Naim in \cite{Krapivsky 2003}, while the dynamics of the nonlinear fragmentation equation with product and sum collision kernels is investigated by Kostoglou \& Karabelas \cite{Kostoglou 2000}, combining analytical solutions and asymptotic expansions, see also \cite{Kostoglou 2006}.  
	
From a mathematical viewpoint, several existence results are available for the continuous collision-induced fragmentation equation when coupled to coagulation, the coagulation being usually assumed to be the dominant mechanism \cite{PKB 2020, PKB 2020I, AKG 2021}. In the absence of coagulation, the existence, non-existence, and uniqueness of mass-conserving solutions to the continuous collision-induced fragmentation equation are investigated in \cite{AKG 2021I} when the collision kernel is of the form $a(x,y) = x^\alpha y^\beta + x^\beta y^\alpha$, $(\alpha,\beta)\in \mathbb{R}^2$. It is shown there that the well-posedness strongly depends on the value of $\alpha+\beta$ and that a finite time singularity may take place, as already observed in \cite{EP 2007} for product collision kernels (corresponding to $\alpha=\beta$).
	
When the size variable ranges in the set of positive integers, the coagulation equation with collisional breakage is explored in Lauren\c{c}ot \& Wrzosek \cite{Laurencot 2001I}, where the existence, uniqueness, mass conservation, and  large time behavior of weak solutions are studied under reasonable restrictions on the collision kernel and the daughter distribution function. The purpose of this work is to go beyond the analysis performed in \cite{Laurencot 2001I} when coagulation is turned off and relax the growth conditions on the collision kernel and the daughter distribution function. More precisely, denoting by $w_i(t)$, $i \in \mathbb{N}$, the number of clusters made of $i$ monomers ($i$-particles) per unit volume at time $t \geq 0$, the discrete collision-induced fragmentation equation reads
\begin{align}
\frac{dw_i}{dt}  =&\frac{1}{2} \sum_{j=i+1}^{\infty} \sum_{k=1}^{j-1} B_{j-k,k}^i a_{j-k,k} w_{j-k} w_k -\sum_{j=1}^{\infty} a_{i,j} w_i w_j,  \hspace{.5cm} i \in \mathbb{N}, \label{NLDCBE}\\
w_i(0) &= w_i^{\rm{in}}, \hspace{.5cm} i \in \mathbb{N},\label{NLDCBEIC}
\end{align}
where $\mathbb{N}$ stands for the set of positive integers. Here $a_{i,j}$ denotes the rate of collisions of $i$-clusters with $j$-clusters and satisfies
\begin{align}
a_{i,j}=a_{j,i} \geq 0,\qquad (i,j) \in \mathbb{N}^2, \label{ASYMM}
\end{align}
 while  $\{B_{i,j}^s, s=1,2,...,i+j-1\}$  is the distribution function of the resulting fragments and satisfies 
 \begin{align}
B_{i,j}^s = B_{j,i}^s \geq 0 \hspace{.7cm} \text{and} \hspace{.7cm} \sum_{s=1}^{i+j-1} s B_{i,j}^s = i+j, \hspace{.5cm} (i,j)\in \mathbb{N}^2. \label{LMC}
\end{align}
The second identity in \eqref{LMC} guarantees that mass is conserved during each collisional breakage event. The first term in \eqref{NLDCBE} takes into account collisions in which a $j$-mer and a $k$-mer collide and form $i$-mers at a rate determined by the breakup kernel $B_{j,k}^i$, whereas the second term accounts for the depletion of $i$-mers due to collisions with other clusters in the system, which occur at a rate determined by the  collision kernel $a_{i,j}$. It is worth pointing out here that the assumption~\eqref{LMC} allows the collision of a $i$-cluster and a $j$-cluster to produce a $i+j-1$-cluster and a $1$-cluster, so that there might be outcoming particles with a larger size than both incoming particles. In other words, mass transfer between the colliding particles may occur and the mean size of the system of particles does not necessarily decrease during the time evolution. This phenomenon is prevented when one considers the discrete counterpart of the model studied in \cite{CHNG 90, EP 2007} which reads

\begin{align}
\frac{dw_i}{dt} =& \sum_{j=i+1}^{\infty} \sum_{k=1}^{\infty}  a_{j,k} b_{i,j;k} w_j w_k -\sum_{j=1}^{\infty} a_{i,j} w_i w_j, \hspace{.5cm} i\in\mathbb{N}, \label{SNLBE}\\
w_i(0) &=w_{i}^{\rm{in}}, \hspace{.5cm} i\in\mathbb{N},\label{SNLBEIC}
\end{align}
where  $\{b_{i,j;k}, 1\leq i \leq j-1\}$ denotes the distribution function of the fragments of a $j$-cluster after a collision with a $k$-cluster, and satisfies the conservation of matter
\begin{align}
\sum_{i=1}^{j-1} i b_{i,j;k} = j, \hspace{.5cm} j\geq 2,~~~ k\geq 1. \label{LMC1}
\end{align}

The rate equation~\eqref{SNLBE} is actually a particular case of~\eqref{NLDCBE},~as easily seen when putting  
\begin{align}
B_{i,j}^s = \textbf{1}_{[s, +\infty)} (i) b_{s,i;j} + \textbf{1}_{[s, +\infty)} (j) b_{s,j;i} \label{NMT}
\end{align}
for $i,j\geq 1$ and $s\in \{1,2,\cdots,i+j-1\},$ where $\textbf{1}_{[s, +\infty)}$ denotes the characteristic function of the interval $[s,+\infty)$. As each cluster splits into smaller pieces after collision it is expected that, in the long time, only 1-clusters remain.\\

In this article, we look for the existence of solutions to \eqref{SNLBE}--\eqref{SNLBEIC} for the class of collision kernels having quadratic growth, i.e.
\begin{align}
 a_{i,j} \leq A_1 ij \hspace{.2cm} \text{for some} \hspace{.2cm} A_1>0 \hspace{.2cm}\text{and}\hspace{.2cm} i,j\geq 1.\label{QUADGROWTH}
\end{align}

In addition  to \eqref{LMC1}, we assume that there are non-negative constants $\beta_0$ and $\beta_1$ such that
\begin{align}
b_{s,i;j} \leq \beta_0 +\beta_1 b_{s,j;i}, \hspace{.2cm} 1\le s \le i-1, \quad j\geq i. \label{bCond}
\end{align}

\begin{remark}
It is worth to mention at this point that \eqref{bCond} includes the class of bounded daughter distribution functions such as 
\begin{align*}
b_{i,j;k} = \frac{2}{j}, \qquad j\geq 2, k\geq1,
\end{align*} 
as well as unbounded daughter distribution functions, e.g.,
\begin{align*}
b_{i,j;k} = j\delta_{i,1}, \qquad j\geq 2, k\geq1.
\end{align*}
 Note that the latter is excluded from the analysis performed in~\cite{Laurencot 2001I}, where the daughter distribution function is assumed to be bounded (i.e., satisfies~\eqref{bCond} with $\beta_1=0$).
\end{remark}

We expect the density $\sum_{i=1}^{\infty} i w_i$ to be conserved because particles are neither generated nor destroyed in the interactions represented by \eqref{SNLBE}--\eqref{SNLBEIC}. This is mathematically equivalent to
\begin{align}
 \sum_{i=1}^{\infty} iw_i(t) = \sum_{i=1}^{\infty} iw_i^{\rm{in}}.\label{MCC}
\end{align}
In other words, the density of the solution $w$ remains constant over time. 

The paper is organized as follows: The next section is devoted to a precise statement of our results, including definitions, the existence of solutions to \eqref{SNLBE}--\eqref{SNLBEIC}, and the mass conservation property of solutions. In Section~\ref{PMCDID}, propagation of moments, uniqueness, and continuous dependence of solutions on initial data are explored, whereas, in Section~\ref{IPOS}, some invariance properties of solutions are shown. Finally, in Section~\ref{LTBOS}, the large-time behaviour of solutions is discussed.

\section{Existence}

\subsection{Main results}
For $\gamma_0 \geq 0$, let $Y_{\gamma_0}$ be the Banach space defined by
\begin{align*}
Y_{\gamma_0} = \Big\{ y =(y_i)_{i\in\mathbb{N}}: y_i \in \mathbb{R}, \sum_{i=1}^{\infty} i^{\gamma_0} |y_i| < \infty \Big\}
\end{align*}
with the norm 
\begin{align*}
\|y\|_{\gamma_0} =\sum_{i=1}^{\infty} i^{\gamma_0} |y_i|, \qquad y\in Y_{\gamma_0}.
\end{align*}

We will use the positive cone $Y_{\gamma_0}^+$ of $Y_{\gamma_0}$, that is,
\begin{align*}
Y_{\gamma_0}^+ =\{y\in Y_{\gamma_0}: ~~y_i \geq 0~~\text{for each}~~ i\geq 1\}.
\end{align*}

It is worth noting that the norm $\|w\|_0$ of a particular cluster distribution $w$ represents the total number of clusters present in the system, and the norm $\|w\|_1$ estimates the overall density or mass of the cluster distribution $w$.

Let us now define what we mean by a solution to \eqref{SNLBE}--\eqref{SNLBEIC}.

\begin{definition} \label{DEF1}
Let $T\in(0,+\infty]$ and $w^{\rm{in}}= (w_{i}^{\rm{in}})_{i \geq 1}\in Y_1^+$ be a sequence of non-negative real numbers. A solution $w=(w_i)_{i \geq 1} $ to \eqref{SNLBE}--\eqref{SNLBEIC} on $[0,T)$ is a sequence of non-negative continuous functions satisfying for each $i\geq 1$ and $t\in(0,T)$ 
\begin{enumerate}
\item $w_i\in \mathcal{C}([0,T))$, $\sum_{j=1}^{\infty}  a_{i,j} w_j \in L^1(0,t)$, $ \sum_{j=i+1}^{\infty}\sum_{k=1}^{\infty}b_{i,j;k} a_{j,k} w_{j} w_k \in L^1(0,t)$,
\item and there holds 
\begin{align}
w_i(t) = w_{i}^{\rm{in}} + \int_0^t \Big( \sum_{j=i+1}^{\infty} \sum_{k=1}^{\infty} b_{i,j;k} a_{j,k} w_{j}(\tau) w_k(\tau) -\sum_{j=1}^{\infty} a_{i,j} w_i(\tau) w_j(\tau) \Big) d\tau. \label{IVOE}
\end{align}
\end{enumerate}
\end{definition}

Our existence result then reads as follows.

\begin{theorem}\label{MAINTHEOREM}
Assume that the collision kernel $(a_{i,j})_{(i,j\in\mathbb{N}^2)}$ satisfies~\eqref{ASYMM} and~\eqref{QUADGROWTH} and that the distribution function satisfies~\eqref{LMC1} and~\eqref{bCond}. Let $w^{\rm{in}}\in Y_1^+$. Then, there is at least one solution $w$ to \eqref{SNLBE}--\eqref{SNLBEIC} on $[0,+\infty)$ satisfying
\begin{align}
\|w(t) \|_{1}= \|w^{\rm{in}}\|_{1},\label{MC}
\end{align} 
and for any $r\geq 1$ and $t>0$,
\begin{align}
\sum_{i=r}^{\infty} i w_i(t) \leq \sum_{i=r}^{\infty} i w_i^{\rm{in}}. \label{TE1}
\end{align}
\end{theorem} 

 As already mentioned, the class of collision kernels and daughter distribution functions included in Theorem~\ref{MAINTHEOREM} is broader than that considered in~\cite{Laurencot 2001I}. More precisely, in~\cite{Laurencot 2001I}, the collision kernel is restricted to subquadratic growth while the daughter distribution function is bounded.
	
We first introduce some notation. We denote by $\mathcal{G}_1$ the set of non-negative and convex functions $G \in C^1([0,+\infty))\cap W_{\text{loc}}^{2,\infty}(0,+\infty)$ such that $G(0)=0$, $G'(0)\geq 0$ and $G'$ is a concave function. We next denote by $\mathcal{G}_{1, \infty}$  the set of functions $G \in \mathcal{G}_1$ satisfying, in addition,
\begin{align}
\lim_{\zeta \to +\infty} G'(\zeta) = \lim_{\zeta\to +\infty} \frac{G(\zeta)}{\zeta} = +\infty.
\end{align}  
\begin{remark} \label{remark1}
It is clear that $\zeta\mapsto \zeta^p$ belongs to $\mathcal{G}_1$ if $p\in [1,2]$ and to $\mathcal{G}_{1,\infty}$ if $p\in (1,2]$. 
\end{remark}

\subsection{Approximating systems}

As in previous works on similar equations, see \cite{BALL 90} for instance, the existence of solutions to \eqref{SNLBE}--\eqref{SNLBEIC} follows by taking a limit of solutions to finite-dimensional systems of ordinary differential equations obtained by truncation of these equations. More precisely, given $l \geq 3$, we consider the following system of $l$ ordinary differential equations

\begin{align}
\frac{dw_i^l}{dt}&=  \sum_{j=i+1}^{l-1} \sum_{k=1}^{l-j}  b_{i,j;k} a_{j,k} w_{j}^l w_k^l -\sum_{j=1}^{l-i} a_{i,j} w_i^l w_j^l,\hspace{.2cm}  \label{FDNLBE} \\
w_i^l(0) &= w_{i}^{\rm{in}} \label{FDNLBEIC}
\end{align}
for $i \in\{1,2, \cdots, l\}$, where the right hand side of \eqref{FDNLBE} is  zero when $i=l$.

Proceeding as in \cite[Lemmas~2.1 and~2.2]{BALL 90}, we obtain the following result.

\begin{lemma}\label{LER}

For $l \geq 3$, the system \eqref{FDNLBE}--\eqref{FDNLBEIC} has a unique solution 
\begin{align*}
w^l= (w_i^l)_{1\leq i \leq l} \in \mathcal{C}^1([0,+\infty);\mathbb{R}^l)
\end{align*}
with $w_i^l(t) \geq 0$ for $1\leq i \leq l$ and $t \geq 0$. Furthermore there holds 
\begin{align}
\sum_{i=1}^l i w_i^l(t) = \sum_{i=1}^l i w_i^{\rm{in}}, \qquad t \in [0,+\infty),\label{TMC}
\end{align}
and if $(\mu_i) \in \mathbb{R}^l$, 
\begin{align} 
\frac{d}{dt}\sum_{i=1}^l \mu_i w_i^l =& \sum_{k=1}^{l-1}\sum_{j=1}^{l-k}\Big( \sum_{i=1}^{k-1} \mu_i b_{i,k;j} -\mu_k\Big)a_{j,k} w_j^l w_k^l. \label{GME}
\end{align}
\end{lemma}

We are now in a position to state and prove the main result of this section.
\begin{prop} \label{prop1}
Consider $G \in \mathcal{G}_1$, then for each $l \geq 3$ and $t\geq 0$, there holds 
\begin{align}
\sum_{i=1}^l G(i) w_i^l(t) \leq \sum_{i=1}^l G(i) w_i^{\rm{in}},  \label{PROPEQN1}
\end{align}

\begin{align}
0 \leq \int_0^t\sum_{j=1}^l \sum_{k=1}^{l-j} \sum_{i=1}^{j-1} \Big(\frac{G(j)}{j}-\frac{G(i)}{i}\Big)i b_{i,j;k}  a_{j,k} w_j^l w_k^l ds \leq \sum_{i=1}^l G(i) w_i^{\rm{in}},\label{PROPEQN2}
\end{align}
and, for $1\leq r\leq l$,
\begin{align}
\sum_{i=r}^{l} iw_i^l(t) \leq \sum_{i=r}^{l} iw_{i}^{\rm{in}}.
\end{align}
\end{prop}

\begin{proof}
For $l \geq 3$ and $t \geq 0$ we put 
\begin{align*}
M_G^l(t) =\sum_{i=1}^l G(i) w_i^l(t).
\end{align*}
It follows from \eqref{LMC1} and \eqref{GME} that

\begin{align}
\frac{d}{dt} M_G^l(t) = &\sum_{j=1}^{l-1} \sum_{k=1}^{l-j} \Big(\sum_{i=1}^{k-1} G(i) b_{i,j;k}-G(j)\Big) a_{j,k} w_{j}^l w_k^l \nonumber\\
& =\sum_{j=1}^{l-1} \sum_{k=1}^{l-j} \sum_{i=1}^{k-1} \Big(G(i)-i\frac{G(j)}{j}\Big) b_{i,j;k}  a_{j,k} w_j^l w_k^l, \nonumber\\
& = -\sum_{j=1}^{l-1} \sum_{k=1}^{l-j} \sum_{i=1}^{j-1} \Big(\frac{G(j)}{j} - \frac{G(i)}{i}\Big) i b_{i,j;k} a_{j,k} w_jw_k. \label{PROPEQN3}
\end{align}

Now, since $G(0)=0$ and $G$ is a convex function, the function $\zeta \mapsto \frac{G(\zeta)}{\zeta}$  is
a non-decreasing function and the term on the right-hand side of \eqref{PROPEQN3} is non-negative.
Therefore 
\begin{align*}
\frac{d}{dt} M_{G}^l (t) \leq 0,
\end{align*}
which yields \eqref{PROPEQN1}. We next integrate \eqref{PROPEQN3} over $(0, t)$ and use the non-negativity of $G$ and $w^l$ to obtain \eqref{PROPEQN2}.

Finally, for $1\leq r\leq l$, we take $\mu_i = i \textbf{1}_{[r,+\infty)}$ in \eqref{GME} and find 
\begin{align*}
\frac{d}{dt}\sum_{i=r}^l i w_i^l(t) =& \sum_{j=r}^{l-1} \sum_{k=1}^{l-j}\Big( \sum_{i=r}^{j-1} i b_{i,j;k} -j\Big)a_{j,k} w_j^l w_k^l 
\end{align*}
Using the condition \eqref{LMC1}, we immediately conclude that,
\begin{align*}
 \frac{d}{dt} \sum_{i=r}^l i w_i^l \leq 0.
 \end{align*}
Then, we have
\begin{align}\label{PMOMNT}
\sum_{i=r}^{l} i w_i^l(t) \leq \sum_{i=r}^l i w_{i}^{\text{in}} \leq \sum_{i=r}^{\infty} i w_{i}^{\text{in}} \leq \sum_{i=1}^{\infty}  iw_{i}^{\text{in}}=\|w^{\text{in}}\|_1,
 	\end{align}
 	and the proof is completed.
\end{proof}

Next, we recall the following result from \cite[Lemma~3.4]{Laurencot 2001I}.
\begin{lemma}
Let $T\in(0,+\infty)$ and $i \geq 1$. There exists a constant $\Pi_i(T)$ depending only on $A$, $\|w^{\rm{in}}\|_{Y_1}$, $i$ and $T$ such that for each $l \geq i$,
\begin{align}
\int_0^T \Big| \frac{dw_i^l}{dt}\Big| d\tau \leq \Pi_i(T). \label{DERVBOUND}
\end{align}
\end{lemma}

\subsection{Existence of a solution}

We are now in a position to prove Theorem~\ref{MAINTHEOREM}. For that purpose we first recall a refined version of the de la Vall\'ee-Poussin theorem for integrable functions \cite[Theorem~7.1.6]{BLL 2019}.

\begin{theorem} \label{DlVPthm}
Let $(\Sigma,\mathcal{A}, \nu)$ be a measured space and consider a function $w \in L^1(\Sigma,\mathcal{A}, \nu)$. Then there exists a function $ G\in \mathcal{G}_{1,\infty}$  such that
\begin{align*}
G(|w|) \in L^1(\Sigma,\mathcal{A}, \nu).
\end{align*}
\end{theorem}

\begin{proof}[Proof of Theorem~\ref{MAINTHEOREM}]
By \eqref{TMC} and \eqref{DERVBOUND}, the sequence $(w_i^l)_{l\geq i}$ is bounded in $L^{\infty}(0,T)\cap W^{1,1}(0,T)$  for each $i \geq 1$ and $T\in(0,+\infty)$. We then infer from the Helly theorem \cite[pp.~372--374]{KF 1970} that there exist a  subsequence of $(w^l)_{l \geq 3}$, still denoted by $(w^l)_{l \geq 3}$, and a sequence $w=(w_i)_{i\geq 1}$ such that
\begin{align}
\lim_{l \to \infty} w_i^l(t) = w_i(t) \label{LIMITw}
\end{align} 
for each $i \geq 1$ and $t\geq 0$. Clearly, $w_i(t)\geq  0$ for $i \geq 1$ and $t \geq 0$ and it follows from \eqref{LIMITw} that, for each $q \in \mathbb{N}$ and $t \geq 0$,
 	\begin{align*}
 	\lim_{l\to \infty}\sum_{i=1}^q iw_i^l(t) = \sum_{i=1}^q i w_i(t).
 	\end{align*}
 In particular, by \eqref{TMC}, for any $q \in \mathbb{N}$ and $t\geq 0$, 
 	\begin{align*}
 	\sum_{i=1}^q iw_i(t) \leq \|w^{\text{in}}\|_{1}. 
 	\end{align*}
 	By letting $q \to \infty$, we obtain
 	\begin{align}
 	\sum_{i=1}^{\infty} iw_i(t)=\|w(t)\|_{1} \leq \|w^{\rm{in}}\|_{1}, \qquad t\geq 0.\label{Y1NORMBOUND}
 	\end{align}

We next apply Theorem~\ref{DlVPthm}, with $\Sigma =\mathbb{N}$ and $\mathcal{A} = 2^{\mathbb{N}}$, the set of all subsets of $\mathbb{N}$. Defining the measure $\nu$ by
\begin{align*}
	\nu(J) =\sum_{i\in J} w_i^{\rm{in}},\qquad J \subset \mathbb{N},
\end{align*}
the condition $w^{in}\in Y_1^+$  ensures that $\zeta \mapsto \zeta$  belongs to $L^1(\Sigma,\mathcal{A}, \nu)$. By Theorem~\ref{DlVPthm} there is thus a function $G_0 \in \mathcal{G}_{1,\infty}$ such that $\zeta \mapsto G_0(\zeta)$  belongs to $L^1(\Sigma,\mathcal{A}, \nu)$; that is,
\begin{align}
	\mathcal{G}_0 =\sum_{i=1}^{\infty} G_0(i) w_i^{\rm{in}} <\infty. \label{Gzero}
\end{align}

  	Furthermore, as $G_0\in  \mathcal{G}_{1,\infty}$, we infer from \eqref{Gzero} and Proposition~\ref{prop1} that, for each $t \geq 0$ and $l\geq 3$, there holds

 	\begin{align}
 	\sum_{i=1}^l G_0(i) w_i^l(t) \leq \mathcal{G}_0, \label{G_0bound1l}
 	\end{align}
 	
 	\begin{align}
 	0 \leq \int_0^t\sum_{j=1}^l \sum_{k=1}^{l-j} \sum_{i=1}^{j-1} (G_1(j)-G_1(i))i b_{i,j;k}  a_{j,k} w_j^l w_k^l ds \leq \mathcal{G}_0, \label{G_0bound2l}
 	\end{align}
 		where 
 	\begin{align*}
 	G_1(\zeta) = \frac{G_0(\zeta)}{\zeta} ~~~ \text{for} ~~~\zeta \geq 0.
 	\end{align*}
 	A consequence of \eqref{G_0bound2l} and the monotonicity properties of $G_1$ is that, for $i \geq 1$, $t\geq 0$ and $l \geq i+1$
 	\begin{align*}
 	0 \le \int_0^t\sum_{j=i+1}^l \sum_{k=1}^{l-j}  (G_1(j)-G_1(i))i b_{i,j;k}  a_{j,k} w_j^l w_k^l ds \leq \mathcal{G}_0.
 	\end{align*}
 	Hence
 	\begin{align}
 	\int_0^t\sum_{j=i+1}^l \sum_{k=1}^{l-j}  (G_1(j)-G_1(i)) b_{i,j;k}  a_{j,k} w_j^l w_k^l ds \leq \frac{\mathcal{G}_0}{i} \leq \mathcal{G}_0. \label{G_0bound3l}
 	\end{align}
 
 	Consider now $t\in (0,+\infty)$ and $m\geq 2$. By \eqref{G_0bound1l}, \eqref{G_0bound2l}, \eqref{G_0bound3l}, and the monotonicity of $G_1$, we have for, $l \geq m+1$.
 	\begin{align*}
 	\sum_{i=1}^m G_0(i) w_i^l(t) \leq \mathcal{G}_0,
 	\end{align*}
\begin{align*}
 	0 \leq \int_0^t\sum_{j=1}^m \sum_{k=1}^{m-j} \sum_{i=1}^{j-1} (G_1(j)-G_1(i))i b_{i,j;k}  a_{j,k} w_j^l w_k^l ds \leq \mathcal{G}_0, 
 	\end{align*}
 	\begin{align*}
 	0 \leq \int_0^t\sum_{j=i+1}^m \sum_{k=1}^{m-j}  (G_1(j)-G_1(i)) b_{i,j;k}  a_{j,k} w_j^l w_k^l ds \leq \mathcal{G}_0. 
 	\end{align*}
 	Due to \eqref{LIMITw} we may pass to the limit as $l\to \infty$ in the above estimates and conclude
that they both hold true with $w_i^l$ replaced by $w_i$. We next let $m \to \infty$ and obtain
\begin{align}
 	\sum_{i=1}^{\infty} G_0(i) w_i(t) \leq \mathcal{G}_0,\label{G_0bound1}
 	\end{align}
\begin{align}
 	0 \leq \int_0^t\sum_{j=1}^{\infty} \sum_{k=1}^{\infty} \sum_{i=1}^{j-1} (G_1(j)-G_1(i))i b_{i,j;k}  a_{j,k} w_j w_k ds \leq \mathcal{G}_0,\label{G_0bound2}
 	\end{align}
 	\begin{align}
 0\leq 	\int_0^t\sum_{j=i+1}^{\infty} \sum_{k=1}^{\infty}  (G_1(j)-G_1(i)) b_{i,j;k}  a_{j,k} w_j w_k ds \leq \mathcal{G}_0.\label{G_0bound3}
 	\end{align}
 	
Since $G_0\in \mathcal{G}_{1,\infty}$, it readily follows from \eqref{LIMITw}, \eqref{G_0bound1l} and \eqref{G_0bound1} that 
\begin{align}
\lim_{l\to \infty} \|w^l(t) - w(t)\|_{1} =0. \label{NORMCONV}
\end{align}
In particular, for $t\geq 0$
\begin{equation*}
\|w(t)\|_{Y_1} =\lim_{l\to \infty} \|w^l(t)\|_{1} =\lim_{l\to \infty} \|w_i^l(0)\|_{1} = \|w^{\rm{in}}\|_{1},
\end{equation*}
so that $w$ satisfies \eqref{MC}.
We then argue  exactly in the same way as in the proof of \cite[Theorem~3.1]{Laurencot 2001I} to show that, for $i\ge 1$,
\begin{align}
\lim_{l\to \infty} \int_0^t \Big | \sum_{j=1}^{l-i} a_{i,j} w_i^l w_j^l -\sum_{j=1}^{\infty} a_{i,j} w_i w_j\Big | d\tau =0. \label{CONV1EQ}
\end{align}

We next turn to the convergence of the first term on the right hand side of \eqref{FDNLBE} and fix $i\geq 1$. For $m \geq i+1$ and $2m<l$, 
\begin{align}
\Bigg|\sum_{j=i+1}^{l} \sum_{k=1}^{l-j} b_{i,j;k} a_{j,k} w_j^l w_k^l-& \sum_{j=i+1}^{\infty} \sum_{k=1}^{\infty} b_{i,j;k} a_{j,k} w_j w_k \Bigg| \nonumber \\
&\leq \Bigg| \sum_{j=i+1}^{m} \sum_{k=1}^m a_{j,k} b_{i,j;k} \big(w_j^l w_k^l - w_jw_k\big)\Bigg| \nonumber\\
&+\sum_{j=m+1}^l \sum_{k=1}^{l-j} b_{i,j;k} a_{j,k} w_j^l w_k^l+ \sum_{j=i+1}^m \sum_{k=m+1}^{l-j}  b_{i,j;k} a_{j,k} w_j^l w_k^l \nonumber\\
& + \sum_{j=i+1}^m \sum_{k=m+1}^{\infty}  b_{i,j;k} a_{j,k} w_j w_k + \sum_{j=m+1}^{\infty} \sum_{k=1}^{\infty} b_{i,j;k} a_{j,k} w_j w_k. \label{CONV3}
\end{align}

 On the one hand, it follows from  \eqref{QUADGROWTH}, \eqref{TMC}, \eqref{LIMITw}, and \eqref{Y1NORMBOUND} and the Lebesgue dominated convergence theorem that
\begin{align}
\lim_{l\to \infty} \int_0^t\Big|\sum_{j=i+1}^m \sum_{k=1}^m b_{i,j;k} a_{j,k} \big(w_j^l w_k^l - w_jw_k\big)\Big|d\tau=0. \label{BD1}
\end{align}

On the other hand we infer from \eqref{G_0bound3l} and the monotonicity of $G_1$ that 
\begin{align}
\int_0^t \sum_{j=m+1}^l \sum_{k=1}^{l-j} b_{i,j;k} & a_{j,k} w_j^l w_k^l d\tau =\int_0^t \sum_{j=m+1}^{l-1} \sum_{k=1}^{l-j} \frac{[G_1(j)- G_1(i)]}{[G_1(j)- G_1(i)]}b_{i,j;k} a_{j,k} w_j^l w_k^l d\tau \nonumber \\
& \leq \frac{1}{[G_1(m+1)- G_1(i)]}\int_0^T\sum_{j=i+1}^{l-1} \sum_{k=1}^{l-j}[G_1(j)- G_1(i)]b_{i,j;k} a_{j,k} w_j^l w_k^l d\tau, \nonumber \\
& \leq \frac{\mathcal{G}_0}{[G_1(m+1)- G_1(i)]}. \label{BD2}
\end{align}
Similarly, using the monotonicity of $G_1$ and \eqref{G_0bound3} gives
\begin{align}
\int_0^t \sum_{j=m+1}^l \sum_{k=1}^{l-j} b_{i,j;k} a_{j,k} &w_j^l w_k^l d\tau \leq \frac{\mathcal{G}_0}{[G_1(m+1)- G_1(i)]}. \label{BD3}
\end{align}
 Using \eqref{QUADGROWTH} and  \eqref{bCond}, we can estimate 
\begin{align*}
\int_0^t \sum_{j=i+1}^m \sum_{k=m+1}^{l-j} b_{i,j;k} a_{j,k}& w_j^l w_k^l d\tau \leq \beta_0 \int_0^t \sum_{j=i+1}^m \sum_{k=m+1}^{l-j} a_{j,k} w_j^l w_k^l d\tau \\
& \quad + \beta_1 \int_0^T \sum_{j=i+1}^m \sum_{k=1}^{l-j} b_{i,k;j} a_{j,k} w_j^lw_k^l d\tau \\
& \leq A_1 \beta_0 \int_0^t \sum_{j=i+1}^m \sum_{k=m+1}^{l-j} jk w_j^l w_k^l d\tau \nonumber\\
& \quad + \beta_1 \int_0^T \sum_{k=i+1}^m \sum_{j=m+1}^{l-k} b_{i,j;k} a_{j,k} w_j^lw_k^l d\tau   \\
 &\leq \beta_0 A_1\int_0^t \sum_{j=i+1}^m \sum_{k=m+1}^{l-j} \frac{jk}{G_0(k)} G_0(k) w_j^l w_k^l d\tau  \\
 & \quad+ \beta_1\int_0^t\sum_{j=m+1}^{l-i-1} \sum_{k=i+1}^{\min\{m,l-j\}}b_{i,j;k} a_{j,k} w_j^l w_k^l d\tau.
 \end{align*}
 
 Hence, by \eqref{G_0bound1l} and \eqref{G_0bound3l}, 
 
 \begin{align}
 \int_0^t \sum_{j=i+1}^m \sum_{k=m+1}^{l-j} b_{i,j;k} a_{j,k}& w_j^l w_k^l ds \leq  A_1\beta_0 t\|w^{\text{in}}\|_1 \frac{\mathcal{G}_0}{G_1(m+1)} \nonumber\\
 & \quad + \beta_1\int_0^t \sum_{j=m+1}^{l-1} \sum_{k=i+1}^{l-j} \frac{[G_1(j)-G_1(i)]}{[G_1(j)-G_1(i)]}b_{i,j;k} a_{j,k} w_j^l w_k^l ds \nonumber\\
 & \leq \frac{\beta_0 A_1t\|w^{\rm{in}}\|_1\mathcal{G}_0}{{G_1(m+1)}}  +\frac{\beta_1\mathcal{G}_0}{[G_1(m+1)- G_1(i)]}. \label{BD4}
\end{align}
Similarly,
\begin{align}
\int_0^T \sum_{j=i+1}^m \sum_{k=m+1}^{\infty} b_{i,j;k} a_{j,k}& w_j w_k ds \leq \frac{\beta_0 A_1t\|w^{\rm{in}}\|_1\mathcal{G}_0}{{G_1(m+1)}}  +\frac{\beta_1\mathcal{G}_0}{[G_1(m+1)- G_1(i)]}. \label{BD5}
\end{align}

Consequently, using \eqref{BD1}, \eqref{BD2}, \eqref{BD3}, \eqref{BD4} and \eqref{BD5} in \eqref{CONV3}, we obtain 

\begin{align*}
\limsup_{l \to \infty} \Bigg|\sum_{j=i+1}^{l} \sum_{k=1}^{l-j} b_{i,j;k} a_{j,k} w_j^l w_k^l- \sum_{j=i+1}^{\infty} \sum_{k=1}^{\infty} b_{i,j;k} a_{j,k} w_j w_k \Bigg| &  \leq \frac{2\beta_0 A_1t\|w^{\rm{in}}\|_1\mathcal{G}_0}{{G_1(m+1)}} \\
 &+\frac{2\beta_1\mathcal{G}_0}{[G_1(m+1)- G_1(i)]}.
\end{align*}
Since $G_0\in \mathcal{G}_{1, \infty}$, we may let $m \to \infty$ to conclude that 
\begin{align}
\lim_{l \to \infty} \Bigg|\sum_{j=i+1}^{l} \sum_{k=1}^{l-j} b_{i,j;k} a_{j,k} w_j^l w_k^l- \sum_{j=i+1}^{\infty} \sum_{k=1}^{\infty} b_{i,j;k} a_{j,k} w_j w_k \Bigg| =0.\label{CONV2EQ}
\end{align}
 	
Owing to \eqref{LIMITw}, \eqref{CONV1EQ} and \eqref{CONV2EQ} we may pass to the limit as $l\to \infty$ in the integral formulation of the equation satisfied by $w_i^l$ and deduce that $w_i$ satisfies~\eqref{IVOE}. Observe that, since the right hand side of \eqref{IVOE} belongs to $L^1_{loc}([0,+\infty))$, $w_i\in \mathcal{C}([0, +\infty))$.  	
\end{proof}

 	In the next section, the issue we consider is that whether, given $w^{\text{in}} \in Y_1^+$  such that $w^{\text{in}} \in Y_\alpha $ for some $\alpha >1 $, the solution $w$ to \eqref{SNLBE}--\eqref{SNLBEIC}  constructed in Theorem~\ref{MAINTHEOREM} enjoys the same properties throughout time evolution; that is,  $ w(t) \in Y_\alpha $ for $t\in (0,+\infty)$.

 	\section{Propagation of moments, Uniqueness and Continuous Dependence on Initial Data}\label{PMCDID}

 	\begin{prop}\label{PMPROP}
 	 Assume that the assumptions \eqref{ASYMM},\eqref{LMC1}, \eqref{QUADGROWTH} and \eqref{bCond} are fullfilled. If $ w^{\rm{in}} \in Y_{\alpha}^+$  for some $\alpha >1$, then the solution $w$  to \eqref{SNLBE}--\eqref{SNLBEIC}  on $[0,+\infty)$ constructed in Theorem~\ref{MAINTHEOREM} satisfies 
 	\begin{equation}\label{ALPHAMOMNT}
 	\sum_{i=1}^{\infty} i^{\alpha} w_i(t) \leq \sum_{i=1}^{\infty} i^{\alpha} w_i^{\rm{in}}, \qquad t \geq 0.
 	\end{equation}
 	\end{prop}

 	\begin{proof}
 	We know from~\eqref{LIMITw} that
 	\begin{align}
 	\lim_{l\to \infty } w_i^l(t) = w_i(t) \label{LIMITwP}
 	\end{align}
 	for each $t \in [0, + \infty)$  and $i \geq 1$, where $w^l$ denotes the solution to \eqref{FDNLBE}--\eqref{FDNLBEIC} given by Lemma~\ref{LER} (this convergence actually only holds for a subsequence but it is irrelevant for the forthcoming proof). On taking $\mu_i= i^{\alpha}$ and $m=1$ in \eqref{GME}, we get
 	 	\begin{align*}
 	 	\frac{d}{dt} \sum_{i=1}^l i^{\alpha} w_i^l =\sum_{k=1}^{l-1} \sum_{j=1}^{l-k} \Big( \sum_{s=1}^{k-1} s^{\alpha} b_{s,k;j} - k^{\alpha} \Big) a_{j,k} w_j^l w_k^l.
        \end{align*} 	
        
 	 	Since $\alpha>1$, the function $s\mapsto s^{\alpha -1}$ is increasing. Hence using \eqref{LMC1}, we have
 	 	$$ \sum_{s=1}^{k-1} s^{\alpha}  b_{s,k;j} \leq k^{\alpha-1} \sum_{s=1}^{k-1} s  b_{s,k;j} =  k^{\alpha}, \qquad k\geq 2, \ j\geq 1.
 	 	$$
 	 	Therefore 
 	 	\begin{align*}
 	 \frac{d}{dt}	\sum_{i=1}^l i^{\alpha} w_i^l \leq 0,
 	 	\end{align*}
 	 	which implies 
 	 	\begin{align*}
 	 	\sum_{i=1}^l i^{\alpha} w_i^l(t) \leq \sum_{i=1}^l i^{\alpha} w_{i}^{\text{in}} \leq \sum_{i=1}^{\infty} i^{\alpha} w_{i}^{\text{in}}\qquad t\geq 0.
 	 	\end{align*}
 	 	With the help of \eqref{LIMITwP}, we may pass to the limit as $l \to \infty$ in the above inequality and obtain 
 	 	\begin{align*}
 	 	\sum_{i=1}^{\infty} i^{\alpha} w_i(t)   \leq \sum_{i=1}^{\infty} i^{\alpha} w_i^{\text{in}}.
 	 	\end{align*}
 	 	This concludes the proof of Proposition~\ref{PMPROP}.
 	\end{proof}

 Next, we put the following stronger assumption on the collision kernel
 	\begin{align}
 	a_{i,j} \leq A_{\gamma} (ij)^{\gamma}, \gamma\in [0,1], \label{AGAMMA}
 	\end{align}
and establish a uniqueness result for \eqref{SNLBE}--\eqref{SNLBEIC}. This will be accomplished as in the usual coagulation-fragmentation equations with the help of Gronwall's inequality. The proof involves slightly more restricted constraints on the collision kernel  and initial condition than those used in the existence result. We begin with a preliminary result concerning continuous dependence for a suitable class of solutions.
 	
 	\begin{prop}\label{UNIQPROP}
 	Assume that the assumptions \eqref{ASYMM}, \eqref{LMC1}, \eqref{bCond} and \eqref{AGAMMA} are fulfilled and let $w$ and $\hat{w}$ be solutions of \eqref{SNLBE}--\eqref{SNLBEIC} with initial conditions  $w(0)= w^{\rm{in}} \in Y_{1+\gamma}^+$ and $\hat{w}(0) = \hat{w}^{\rm{in}} \in Y_{1+\gamma}^+$  such that 
 	\begin{equation}
 		w\in L^{\infty}((0,T),Y_{1+\gamma}^+) \quad\text{ and }\quad\hat{w}\in L^{\infty}((0,T),Y_{1+\gamma}^+) \label{Z}
 	\end{equation} for each $T>0$. Then, for each $T\geq 0$, there is a positive $\kappa(T,\|w^{\rm{in}}\|_{1+\gamma},\|\hat{w}^{\rm{in}}\|_{1+\gamma})$ such that 
 	
 \begin{align}
 \sup_{t\in [0,T]} \|w(t)- \hat{w}(t) \|_1 \leq \kappa(T,\|w^{\rm{in}}\|_{1+\gamma},\|\hat{w}^{\rm{in}}\|_{1+\gamma} )\|w^{\rm{in}}- \hat{w}^{\rm{in}} \|_1.\label{CD1}
 \end{align}

 	\end{prop}
 	\begin{proof}

 	For $i \geq 1$ and $t\geq 0$, we put
 	\begin{align*}
 	\eta_i(t) = w_i(t)- \hat{w}_i(t) \hspace{.4cm} \text{and}  \hspace{.4cm}\theta_i(t) = \sgn(\eta_i(t)),
 	\end{align*}
 	where $\sgn(h) = h /|h| $ if $ h\in \mathbb{R}\setminus \{0\}$ and $\sgn(0) =0$.
 Now,  for $l\geq 3$, we infer from \eqref{IVOE} that 
 \begin{align}
 \sum_{i=1}^l i |\eta_i(t)| = \int_0^t \sum_{m=1}^3 \Delta_m^l(s) ds, \label{DEL}
 \end{align}
 where 
 \begin{align}
 \Delta_1^l =  \sum_{j=1}^l \sum_{k=1}^{l} \Big( \sum_{i=1}^{j-1}  i\theta_i b_{i,j;k}-  j\theta_j\Big) a_{j,k} (w_j w_k -\hat{w}_j\hat{w}_k),
  \end{align}
 
 \begin{align}
 \Delta_2^l =  \sum_{j=1}^l \sum_{k=l+1}^{\infty} \Big( \sum_{i=1}^{j-1}  i\theta_i b_{i,j;k}-  j\theta_j\Big) a_{j,k} (w_j w_k -\hat{w}_j\hat{w}_k),
  \end{align}

 \begin{align}
 \Delta_3^l = \sum_{j= l+1}^{\infty}\sum_{k=1}^{\infty}\sum_{i=1}^l  i \theta_i b_{i,j;k} a_{j,k} (w_jw_k - \hat{w}_j\hat{w}_k),
 \end{align}

The first term $\Delta_1^l$ can be estimated as follows
\begin{align*}
\Delta_1^l = &\sum_{j=1}^{l} \sum_{k=1}^{l}  \Big(\sum_{i=1}^{j-1}i\theta_i b_{i,j;k}-  j\theta_j\Big)a_{i,j} (\eta_j w_k + \hat{w}_j \eta_k)\\
 &\leq \sum_{j=1}^{l} \sum_{k=1}^{l}  \Big(\sum_{i=1}^{j-1}i b_{i,j;k}-  j\Big)a_{i,j} |\eta_j|w_k\\
 &+ \sum_{j=1}^{l} \sum_{k=1}^{l}  \Big(\sum_{i=1}^{j-1}i b_{i,j;k}+ j \Big)a_{i,j} \hat{w}_j |\eta_k|.
\end{align*}
Using \eqref{LMC1} and \eqref{AGAMMA}, we obtain 
\begin{align}
\Delta_1^l & \leq 2 A_{\gamma} \Big(\sum_{k=1}^{l}k|\eta_k|\Big) \Big(\sum_{j=1}^{l} j^{1+\gamma} \hat{w}_j\Big) \nonumber\\
& \leq 2 A_{\gamma} \Big(\sum_{k=1}^{\infty} k|\eta_k|\Big) \Big(\sum_{j=1}^{\infty} j^{1+\gamma} \hat{w}_j\Big). \label{DEL1}
\end{align}

Next, we deduce from \eqref{AGAMMA} that
\begin{align*}
\int_0^t \Bigg| \sum_{j=1}^l \sum_{k=l+1}^{\infty} \Big(\sum_{i=1}^{j-1}i b_{i,j;k}-  j\Big) a_{j,k} w_j w_k \Bigg| ds & \leq 2A_{\gamma} \int_0^t \Big(\sum_{j=1}^l j^{1+\gamma} w_j\Big) \Big(\sum_{k=l+1}^{\infty}k^{\gamma}w_k\Big)ds\\
& \leq 2A_{\gamma} \int_0^t \Big(\sum_{j=1}^l j^{1+\gamma} w_j\Big) \Big(\sum_{k=l+1}^{\infty}k w_k\Big)ds.
\end{align*}
Using~\eqref{TE1} and~\eqref{Z}, we obtain
\begin{align*}
\lim_{l \to +\infty} \int_0^t \Bigg| \sum_{j=1}^l \sum_{k=l+1}^{\infty} \Big(\sum_{i=1}^{j-1}i b_{i,j;k}-  j\Big) a_{j,k} w_j w_k \Bigg| ds =0,
\end{align*} 
from which we conclude that
\begin{align}
\lim_{l\to \infty} \Delta_2^l  =0. \label{DEL2}
\end{align}

In a similar vein, we can show that

\begin{align}
\lim_{l\to \infty} \Delta_3^l  =0. \label{DEL3}
\end{align}

On substituting \eqref{DEL1}, \eqref{DEL2} and \eqref{DEL3} into \eqref{DEL}, and after passing to the limit as $l\to\infty$, we arrive at
 \begin{align*}
 \sum_{i=1}^{\infty} i |\eta_i(t)|  \leq & 2A_{\gamma} \int_0^t \Big( \sum_{i=1}^{\infty} i |\eta_i(s)|\Big)\Big(\sum_{j=1}^{\infty} j^{1+\gamma} w_j(s)\Big) ds.
\end{align*} 
 
 Finally, we use Gronwall's lemma to complete the proof of Proposition~\ref{UNIQPROP}.
 
 \end{proof}

\begin{corollary} \label{COR1}
 	Assume that the assumptions \eqref{ASYMM}, \eqref{LMC1}, \eqref{bCond} and \eqref{AGAMMA} are fulfilled. Given $w^{\rm{in}}\in Y_{1+\gamma}^+$, there is a unique solution $w$ to \eqref{SNLBE}--\eqref{SNLBEIC} on $[0,+\infty)$ satisfying
	\begin{align}
		\sup_{t\in [0,T]} \sum_{i=1}^{\infty} i^{1+\gamma} w_i(t) <\infty \label{GAMMMNT}
 	\end{align}
	for each $T\in(0,+\infty)$.
\end{corollary}
 	\begin{proof}
 	As $\gamma \in  [0, 1]$ it follows from \eqref{AGAMMA} that $a_{i,j}$  satisfy \eqref{QUADGROWTH}
and the existence of a solution to \eqref{SNLBE}--\eqref{SNLBEIC} on $[0,+\infty)$ with the properties stated in Corollary~\ref{COR1} is a consequence of Theorem~\ref{MAINTHEOREM} and Proposition~\ref{PMPROP}.  That it is unique is guaranteed by Proposition~\ref{UNIQPROP}.
 	
 	\end{proof}

 In the following section, we will demonstrate that when~\eqref{bCond} holds with $\beta_1=0$, that is, the daughter distribution function is bounded, then solutions to \eqref{SNLBE}--\eqref{SNLBEIC} are first-order differentiable.

 \section{Differentiability of the solutions}\label{DOS}
 
 We first establish the following preparatory result, providing additional continuity properties of the solution to~\eqref{SNLBE}--\eqref{SNLBEIC} constructed in Theorem~\ref{MAINTHEOREM}.
 
 \begin{lemma}\label{LEMZ}
 Let $w^{\rm{in}}\in Y_1^+$ and $w$ be a solution to \eqref{SNLBE}--\eqref{SNLBEIC} on $[0,T)$ in the sense of Definition~\ref{DEF1} satisfying additionally 
 \begin{align}
\sum_{i=r}^{\infty} i w_i(t) \leq \sum_{i=r}^{\infty} iw_i^{\rm{in}}, \qquad r\in\mathbb{N}, \ t\in [0,T). \label{TAILBND}
 \end{align}
 Then $w\in \mathcal{C}([0,T), Y_1)$.\\
 Assume further that $a_{i,j}$ satisfies~\eqref{QUADGROWTH}. Then for $i\geq 1$, 
 \begin{align*}
 \sum_{j=1}^{\infty} a_{i,j} w_j \in \mathcal{C}([0,T)).
 \end{align*}
 \end{lemma}
 \begin{proof}
 For $(t,s)\in [0,T)^2$ and $r\geq 1$, 
 \begin{align*}
 \|w(t)-w(s)\|_{1} \leq& \sum_{i=1}^r i \big|w_i(t) -w_i(s)\big| +\sum_{i=r+1}^{\infty} i (w_i(t)+ w_i(s))\\
 &\leq \sum_{i=1}^r i \big|w_i(t) - w_i(s) | + 2 \sum_{i=r+1}^{\infty} i w_i^{\rm{in}}.
 \end{align*}
 Since $w_i$ is continuous for $i\in \{1,2,\cdots,r\}$, we deduce from the above inequality that 
 \begin{align*}
 \limsup_{s\to t} \|w(t)-w(s) \|_{1} \leq 2 \sum_{i=r+1}^{\infty} i w_i^{\rm{in}}.
 \end{align*}
 The above upper bound being valid for any $r\geq 1$, we take the limit $r\to \infty$ to conclude that 
 \begin{align*}
 \lim_{s\to t} \|w(t) - w(s)\|_{1}=0.
 \end{align*}

 Next, for $0\leq s\leq t <T$ and $i \geq 1$, 
 \begin{align*}
 \Bigg| \sum_{j=1}^{\infty} a_{i,j} w_j(t) - \sum_{j=1}^{\infty} a_{i,j} w_j(t)\Bigg|\leq A_1 i \|w(t) - w(s)\|_{1},
 \end{align*}
 from which the time continuity of $\sum_{j=1}^{\infty}a_{i,j}w_j$ follows.
 \end{proof}

 	\begin{prop}
 	 Assume that the assumptions~\eqref{ASYMM}, \eqref{QUADGROWTH} and~\eqref{LMC1} are fulfilled and suppose that the assumption~\eqref{bCond} holds true with $\beta_1=0$. Let $w^{\rm{in}}\in Y_1^+$ and consider the solution $w=(w_i)_{i\ge 1}$ to \eqref{SNLBE}--\eqref{SNLBEIC} on $[0,+\infty)$ given by Theorem~\ref{MAINTHEOREM}. Then $w_i$ is continuously differentiable on $[0,+\infty)$ for each $i \in \mathbb{N}$. 
 	\end{prop}

 	\begin{proof}
 	
 	As the time continuity of the second term of~\eqref{SNLBE} follows from Lemma~\ref{LEMZ}, it is enough to show the time continuity of the first term to complete the proof. To this end, we note that, for $i \geq 1$, $0 \leq s\leq t$, 
 	
 	\begin{align*}
 	\Bigg| &\sum_{j=i+1}^{\infty} \sum_{k=1}^{\infty} a_{j,k} b_{i,j;k} \big[w_j(t)w_k(t) -w_j(s) w_k(s)\big]\Bigg|\\
 	& \quad \leq A_1 \beta_0 \sum_{j=i+1}^{\infty} \sum_{k=1}^{\infty} jk \Big(\big|w_j(t) -w_j(s)\big|w_k(s) + \big|w_k(s)- w_k(t)\big| w_j(t) \Big)\\
 	& \quad \leq A_1 \beta_0 \|w(t) - w(s) \|_{1} \big(\|w(s)\|_{1}	+ \|w(t)\|_{1}\big). 
 	\end{align*}
 	
 	This clearly shows that the right-hand side of~\eqref{SNLBE} is continuous in time, implying the continuity of the derivative of $w_i$. Hence, this guarantees the existence of a classical solution.
 	\end{proof}

 	\section{Some Invariance properties of solutions}\label{IPOS}

 	It is natural to predict that under no mass transfer condition~\eqref{NMT}, if there are no clusters larger than $m$ at the beginning of the physical process, then none will be generated afterwards. This will be established in the next proposition.

 	\begin{prop}
 		Assume that the assumptions~\eqref{ASYMM}, \eqref{LMC1},  \eqref{QUADGROWTH} and~\eqref{bCond} are fulfilled and that there is $m\in\mathbb{N}$ such that $w^{in}\in Y_1^{\sharp m}$, where
 			\begin{align*}
 				Y_1^{\sharp m} := \{w \in Y_1^+ | w_i=0,~~~ \forall i>m \}.
 			\end{align*}
 		Then the solution $w$ to~\eqref{SNLBE}--\eqref{SNLBEIC} given by Theorem~\ref{MAINTHEOREM} satisfies $w(t)\in Y_1^{\sharp m}$ for all $t\ge 0$. Equivalently, for every $m\in \mathbb{N}$, the sets $Y_1^{\sharp m}$ are positively invariant for \eqref{SNLBE}--\eqref{SNLBEIC}.
 	\end{prop} 
 	\begin{proof}
 	 It follows from~\eqref{TE1} that, when the system of particles has no cluster of size greater than $m$ initially, then no cluster of size  greater than $m$ appears  at any further time, which proves the stated assertion.
 	\end{proof}
 	
 	This invariance condition also appears in linear fragmentation equations: if the initial cluster distribution contains no cluster of size  larger than $m$, then none can be formed by fragmentation of the (smaller) ones that are already there.
 	
 	\par
 	In the upcoming section, we will discuss the large time behaviour of the solution, and our result follows the proof of \cite[Proposition~4.1]{Laurencot 2001I}, where it has been proved for collision kernels having linear growth.

 	\section{On the large-time behaviour of solutions} \label{LTBOS}

 	The investigation of the large time behaviour of solutions is studied in this section. Owing to~\eqref{NMT}, as previously stated, a cluster only forms smaller fragments after colliding. As a result, we anticipate that only $1$-clusters will be left in the long time.
 	
 	\begin{prop}
 	 Assume that the assumptions~\eqref{ASYMM}, \eqref{LMC1}, \eqref{QUADGROWTH} and~\eqref{bCond} are fulfilled and consider $w^{in}\in Y_1^+$. Let $w$ be the solution to~\eqref{SNLBE}--\eqref{SNLBEIC} given by Theorem~\ref{MAINTHEOREM}. Then there is $w^{\infty} = (w_i^{\infty})_{i\ge 1}\in Y_1^+$ such that
 	\begin{align}
 	\lim_{t \to \infty} \|w(t) - w^{\infty}\|_{1} = 0 . \label{WINFTYLIM}
 	\end{align}
 	Moreover, if $i \geq 2$ is such that $a_{i,i}\neq 0 $ we have
 	\begin{align}
 	w_i^{\infty} = 0. \label{WINFZERO}
 	\end{align}
 	\end{prop}
 	\begin{remark}
 	In particular, if $a_{i,i} >0$ for each $i \geq 2$, then $w_i^{\infty} = 0 $ for every $i\geq 2$,  and the mass conservation~\eqref{MC} and~\eqref{WINFTYLIM} entail that $w_1^{\infty} = \|w^{\rm{in}} \|_{1}$.
 	\end{remark}

 	\begin{proof}
 	The proof follows exactly the same lines as that of \cite[Proposition~4.1]{Laurencot 2001I}, see also \cite{ZHENG 2005}.
 	\end{proof}
 
 	\subsection*{Funding} The work of the authors was partially supported by the Indo-French Centre for Applied Mathematics (MA/IFCAM/19/58) within the project Collision-Induced Fragmentation and Coagulation:
Dynamics and Numerics. AKG wishes to thank Science and Engineering Research Board (SERB), Department of Science and Technology (DST), India, for their funding support through the MATRICS  project MTR/2022/000530 for completing this work. MA would like to thank the University Grant Commission (UGC),
India for granting the Ph.D. fellowship through Grant No. 416611.
 	\subsection*{Acknowledgements}  Part of this work was done while MA  and AKG enjoyed the hospitality of Laboratoire de Math\'ematiques (LAMA), Universit\'e Savoie Mont Blanc, Chamb\'ery, France.

\end{document}